\documentclass{article}
\usepackage{amsmath, amsfonts, amsthm, amssymb}

\usepackage{times}

\usepackage{enumitem}

\usepackage{graphicx} 
\usepackage[labelformat=simple]{subcaption}

\usepackage{natbib}

\usepackage{algorithm}
\usepackage{algorithmic}
\usepackage{bbm}

\usepackage{hyperref}

\usepackage[accepted]{icml2017}

\icmltitlerunning{Nonparanormal Information Estimation}

\newtheorem{theorem}{Theorem}
\newtheorem{lemma}[theorem]{Lemma}
\newtheorem{proposition}[theorem]{Proposition}

\newtheorem{corollary}[theorem]{Corollary}
\newtheorem{definition}[theorem]{Definition}

\renewenvironment{proof}{{\bf Proof:}}{\hfill\rule{2mm}{2mm}}

\newcommand{\inv}{^{-1}}                            
\newcommand{\N}{\mathbb{N}}                         
\newcommand{\R}{\mathbb{R}}                         
\newcommand{\e}{\varepsilon}                        
\newcommand{\E}{\mathop{\mathbb{E}}}                
\newcommand{\Var}{\mathbb{V}}                       
\newcommand{\Cov}{\operatornamewithlimits{Cov}}     
\newcommand{\Corr}{\operatornamewithlimits{Corr}}    
\newcommand{\pr}{\mathbb{P}}                        
\newcommand{\sign}{\operatorname{sign}}             

\renewcommand{\phi}{\varphi} 

\renewcommand{\hat}{\widehat}
\renewcommand{\tilde}{\widetilde}

\begin{document}

\twocolumn[
\icmltitle{Nonparanormal Information Estimation}
\title{Nonparanormal Information Estimation}
\icmlauthor{Shashank Singh}{sss1@andrew.cmu.edu}
\icmladdress{Carnegie Mellon University,
             5000 Forbes Ave., Pittsburgh, PA 15213 USA}
\icmlauthor{Barnab\'as P\'oczos}{bapoczos@cs.cmu.edu}
\icmladdress{Carnegie Mellon University,
             5000 Forbes Ave., Pittsburgh, PA 15213 USA}
\icmlkeywords{mutual information,entropy,nonparanormal,Gaussian copula,Spearman,Kendall}
\vskip 0.3in]

\begin{abstract}
We study the problem of using i.i.d. samples from an unknown multivariate probability distribution $p$ to estimate the mutual information of $p$. This problem has recently received attention in two settings: (1) where $p$ is assumed to be Gaussian and (2) where $p$ is assumed only to lie in a large nonparametric smoothness class. Estimators proposed for the Gaussian case converge in high dimensions when the Gaussian assumption holds, but are brittle, failing dramatically when $p$ is not Gaussian. Estimators proposed for the nonparametric case fail to converge with realistic sample sizes except in very low dimensions. As a result, there is a lack of robust mutual information estimators for many realistic data. To address this, we propose estimators for mutual information when $p$ is assumed to be a nonparanormal (a.k.a., Gaussian copula) model, a semiparametric compromise between Gaussian and nonparametric extremes. Using theoretical bounds and experiments, we show these estimators strike a practical balance between robustness and scaling with dimensionality.
\end{abstract}

\section{Introduction}
\label{sec:introduction}

This paper is concerned with the problem of estimating entropy or mutual information of an unknown probability density $p$ over $\R^D$, given $n$ i.i.d. samples from $p$. Entropy and mutual information are fundamental information theoretic quantities, and consistent estimators for these quantities have a host of applications within machine learning, statistics, and signal processing. For example, entropy estimators have been used for goodness-of-fit testing~\citep{goria05new},  parameter estimation in semi-parametric models~\citep{Wolsztynski85minimum}, texture classification and image registration~\citep{hero01alpha,hero02applications},  change point detection~\citep{bercher00estimating}, and anomaly detection in networks~\citep{noble03graphAnomaly,nychis08empirical,berezinski15entropy}. Mutual information is a popular nonparametric measure of dependence, whose estimators have been used in feature selection~\citep{peng05feature,shishkin16efficient}, clustering~\citep{aghagolzadeh07hierarchical}, learning graphical models~\citep{chow68chowLiuTree}, fMRI data processing~\citep{chai09fMRI}, prediction of protein structures~\citep{adami04information}, boosting and facial expression recognition~\citep{shan05conditional}, and fitting deep nonlinear models~\citep{hunter16fittingDeepNonlinearModels}. Estimators for both entropy and mutual information have been used in independent component and subspace analysis \citep{radical03,szabo07undercomplete_TCC}.

Motivated by these and other applications, several very recent lines of work (discussed in \hyperref[sec:related_work]{Section~\ref{sec:related_work}}) have studied information estimation,\footnote{We will collectively call the closely related problems of entropy and mutual information estimation \emph{information estimation}.} focusing largely on two settings:
\begin{enumerate}[wide,noitemsep,topsep=0mm]
\item
{\bf Gaussian Setting:} If $p$ is known to be Gaussian, there exist information estimators with mean squared error (MSE) at most $2\log \left(1 - \frac{D}{n} \right)$ and an (almost matching) minimax lower bound of $2D/n$ \citep{cai15logDetCov}.
\item
{\bf Nonparametric Setting:} If $p$ is assumed to lie in a nonparametric smoothness class, such an $s$-order\footnote{Here, $s$ encodes the degree of smoothness, roughly corresponding to the number of continuous derivatives of $p$.} H\'older or Sobolev class, then the minimax MSE is of asymptotic order $\asymp \max \left\{ n\inv, n^{-\frac{8s}{4s + D}} \right\}$ \citep{birge95densityFunctionals}.
\end{enumerate}

In the Gaussian setting, consistent estimation is tractable even in the high-dimensional case where $D$ increases fairly quickly with $n$, as long as $D/n \to 0$. However, optimal estimators for the Gaussian setting rely heavily on the assumption of joint Gaussianity, and their performance can degrade quickly when the data deviate from Gaussian. Especially in high dimensions, it is unlikely that data are jointly Gaussian, making these estimators brittle in practice. In the nonparametric setting, the theoretical convergence rate decays exponentially with $D$, and, it has been found empirically that information estimators for this setting fail to converge at realistic sample sizes in all but very low dimensions. Also, most nonparametric estimators are sensitive to tuning of bandwidth parameters, which is challenging for information estimation, since no empirical error estimate is available for cross-validation.

Given these factors, though the Gaussian and nonparametric cases are fairly well understood in theory, there remains a lack of practical information estimators for the common case where data are neither exactly Gaussian nor very low dimensional. The {\bf main goal of this paper} is to fill the gap between these two extreme settings by studying information estimation in a semiparametric compromise between the two, known as the ``nonparanormal'' (a.k.a. ``Gaussian copula'') model (see \hyperref[def:nonparanormal]{Definition \ref{def:nonparanormal}}). The nonparanormal model, analogous to the additive model popular in regression~\citep{friedman81projection}, limits complexity of interactions among variables but makes minimal assumptions on the marginal distribution of each variable. The result scales better with dimension than nonparametric models, while being more robust than Gaussian models.

{\bf Paper Organization:} \hyperref[sec:problem_statement]{Section~\ref{sec:problem_statement}} gives definitions and notation to formalize the nonparanormal information estimation problem. \hyperref[sec:related_work]{Section~\ref{sec:related_work}} discusses the history of the nonparanormal model and prior work on information estimation, motivating our contributions. \hyperref[sec:estimators]{Section~\ref{sec:estimators}} proposes three estimators, while \hyperref[sec:main_results]{Section~\ref{sec:main_results}} presents our theoretical error bounds, proven in the Appendix. 
\hyperref[sec:empirical]{Section~\ref{sec:empirical}} provides simulation results. While most of the paper discusses mutual information estimation, \hyperref[sec:entropy]{Section~\ref{sec:entropy}} discusses additional considerations arising in entropy estimation.  \hyperref[sec:conc_and_future]{Section \ref{sec:conc_and_future}} presents some concluding thoughts and avenues for future work.

\section{Problem statement and notation}
\label{sec:problem_statement}
There are a number of distinct generalizations of mutual information to more than two variables. The definition we consider is simply the difference between the sum of marginal entropies and the joint entropy:
\begin{definition}
{\bf (Multivariate mutual information)}
Let $X_1,\dots,X_D$ be $\R$-valued random variables with a joint probability density
$p~:~\R^D~\to~[0, \infty)$ and marginal densities $p_1,...,p_D : \R \to [0, \infty)$. The \emph{multivariate mutual information
$I(X)$ of $X = (X_1,\dots,X_D)$} is defined by
\begin{align}
\notag
I(X)
& := \E_{X \sim p} \left[
        \log \left(
             \frac{p(X)}
                  {\prod_{j = 1}^D p_j(X_j)}
        \right)
      \right] \\
\label{eq:entropy_relation}
& = \sum_{j = 1}^D H(X_j) - H(X),
\end{align}
where $H(X) = -\E_{X \sim p} [\log p(X)]$ denotes entropy of $X$.
\end{definition}

This notion of multivariate mutual information, originally due to \citet{watanabe60totalInformation} (who called it ``total correlation'') measures total dependency, or redundancy, within a set of $D$ random variables. It has also been called the ``multivariate constraint''~\citep{garner62multivariateConstraint} and ``multi-information''~\citep{studeny98multiinformation}. Many related information theoretic quantities can be expressed in terms of $I(X)$, and can thus be estimated using estimators of $I(X)$. Examples include pairwise mutual information $I(X,Y) = I((X,Y)) - I(X) - I(Y)$,
which measures dependence between (potentially multivariate) random variables $X$ and $Y$, conditional mutual information
\[I(X|Z) = I((X,Z)) - \sum_{j = 1}^D I((X_j,Z)),\]
which is useful for characterizing how much dependence within $X$ can be explained by a latent variable $Z$ \citep{studeny98multiinformation}, and transfer entropy (a.k.a. directed information) $T_{X~\to~Y}$, which measures predictive power of one time series $X$ on the future of another time series $Y$. $I(X)$ is also related to entropy via Eq.~\eqref{eq:entropy_relation}, but, unlike the above quantities, this relationship depends on the marginal distributions of $X$, and hence involves some additional considerations, as discussed in \hyperref[sec:entropy]{Section~\ref{sec:entropy}}.

We now define the class of nonparanormal distributions, from which we assume our data are drawn.
\begin{definition}
{\bf (Nonparanormal distribution, a.k.a. Gaussian copula model)}
A random vector $X = (X_1,\dots,X_D)^T$ is said to have a \emph{nonparanormal
distribution} (denoted $X \sim \mathcal{NPN}(\Sigma; f)$) if there exist functions $\{f_j\}_{j = 1}^D$ such that each
$f_j : \R \to \R$ is a diffeomorphism
\footnote{A diffeomorphism is a continuously differentiable bijection $g : \R \to R \subseteq \R$ such that $g\inv$ is continuously differentiable.
}
and $f(X) \sim \mathcal{N}(0, \Sigma)$, for some (strictly) positive definite $\Sigma \in \R^{D \times D}$ with $1$'s on the diagonal (i.e., each $\sigma_j = \Sigma_{j,j} = 1$).
\footnote{Setting $\E \left[ f(X) \right] = 0$ and each $\sigma_j = 1$ ensures model identifiability, but does not reduce the model space, since these parameters can be absorbed into the marginal transformation $f$.} $\Sigma$ is called the \emph{latent covariance} of $X$ and $f$ is called the \emph{marginal transformation} of $X$.
\label{def:nonparanormal}
\end{definition}

The nonparanormal family relaxes many constraints of the Gaussian family. Nonparanormal distributions can be multi-modal or heavy-tailed, can encode noisy nonlinear dependencies amongst variables, and need not be supported on $\R^D$. Assumptions made by a nonparanormal model on the marginals are minimal; any desired continuously differentiable marginal cumulative distribution function (CDF) $F_i$ of the variable $X_i$ corresponds to the marginal transformation $f_i(x) = \Phi\inv(F_i(x))$ (where $\Phi$ is the standard normal CDF). As examples, for a Gaussian variable $Z$, the $2$-dimensional case, $X_1\sim\mathcal{N}(0,1)$, and $X_2 = T(X_1 + Z)$ is completely captured by a Gaussian copula when $T(x) = x^3$, $T = \tanh$, $T = \Phi$, or any other diffeomorphism. On the other hand, the limits of the Gaussian copula appear, for example, when $T(x) = x^2$, which is not bijective; then, if $\E[Z] = 0$, the Gaussian copula approximation of $(X_1,X_2)$ will model $X_1$ and $X_2$ as independent.

We are now ready to formally state our problem:

{\bf Formal Problem Statement:}
\emph{Given $n$ i.i.d. samples $X_1,...,X_n \sim \mathcal{NPN}(\Sigma;f)$, where $\Sigma$ and $f$ are both unknown, we would like to estimate $I(X)$.}

{\bf Other notation:} $D$ denotes the dimension of the data (i.e., $\Sigma \in \R^{D \times D}$ and $f : \R^D \to \R^D$). For a positive integer $k$, $[k] = \{1,...,k\}$ denotes the set of positive integers less than $k$ (inclusive). For consistency, where possible, we use $i \in [n]$ to index samples and $j \in [D]$ to index dimensions (so that, e.g., $X_{i,j}$ denotes the $j^{th}$ dimension of the $i^{th}$ sample). Given a data matrix $X \in \R^{n \times D}$, our estimators depend on the empirical rank matrix
\begin{equation}
R \in [n]^{n \times D}
\quad \text{ with } \quad
R_{i,j} := \sum_{k = 1}^n 1_{\{X_{i,j} \geq X_{k,j} \}}.
\label{def:empirical_rank_matrix}
\end{equation}

For a square matrix $A \in \R^{k \times k}$, $|A|$ denotes the determinant of $A$, $A^T$ denotes the transpose of $A$, and
\[\|A\|_2 := \max_{\scriptsize\shortstack{$x \in \R^k$ \\ $\|x\|_2 = 1$}} \|Ax\|_2
\quad \text{ and } \quad
\|A\|_F := \sqrt{\sum_{i, j \in [k]} A_{i,j}^2}\]
denote the spectral and Frobenius norms of $A$, respectively. When $A$ is symmetric, $\lambda_1(A) \geq \lambda_2(A) \geq \cdots \geq \lambda_D(A)$ are its eigenvalues.

\section{Related Work and Our Contributions}
\label{sec:related_work}
\subsection{The Nonparanormal}

Nonparanormal models have been used for modeling dependencies among high-dimensional data in a number of fields, such as graphical modeling of gene expression data~\citep{liu12SKEPTIC}, of neural data~\citep{berkes09neuralDependencies}, and of financial time series~\citep{malevergne03testing,wilson10copula,hernandez13gaussianCopulaFinancial}, extreme value analysis in hydrology~\citep{renard07hydrology,aghakouchak14entropy}, and informative data compression~\citep{rey12metaGaussianIB}.
Besides being more robust generalizations of Gaussians, nonparanormal distributions are also theoretically motivated in certain contexts. For example, the output $Z$ of a neuron is often modeled by feeding a weighted linear combination $Y = \sum_{k = 1}^N w_k X_k$ of inputs into a nonlinear transformation $Z = f(Y)$. When the components of $X$ are independent, the central limit theorem suggests $Y$ is approximately normally distributed, and hence $Z$ is approximately nonparanormally distributed~\citep{szabo07post}.

With one recent exception~\citep{ince16I_G}, previous information estimators for the nonparanormal case~\citep{calsaverini09copulaInformation,ma11copulaEntropy,elidan13copulas}, rely on fully nonparametric information estimators as subroutines, and hence suffer strongly from the curse of dimensionality. Very recently, \citet{ince16I_G} proposed what we believe is the first mutual information estimator tailored specifically to the nonparanormal case; their estimator is equivalent to one of the estimators ($I_G$, described in Section~\ref{subsec:I_G}) we study. However, they focused on its applications to neuroimaging data analysis, and did not study its performance theoretically or empirically.

\subsection{Information Estimation}
Our motivation for studying the nonparanormal family comes from trying to bridge two recent approaches to information estimation. The first has studied fully non-parametric entropy estimation, assuming only that data are drawn from a smooth probability density $p$, where smoothness is typically quantified by a H\"older or Sobolev exponent $s \in (0, \infty)$, roughly corresponding to the continuous differentiability of $s$.
In this setting, the minimax optimal MSE rate has been shown by \citet{birge95densityFunctionals} to be $O \left( \max \left\{ n\inv, n^{-\frac{8s}{4s + D}} \right\} \right)$.
This rate slows exponentially with the dimension $D$, and, while many estimators have been proposed \citep{pal10RenyiEntropyEstimation,sricharan10entropyConfidence,sricharan2013ensemble,singh2014exponential,singh2014generalized,krishnamurthy14Renyi,moon14fDivergence,moon14fDivergenceConfidence,singh16KNNFunctionals,moon17ensembleMI} for this setting, their practical use is limited to a few dimensions\footnote{``Few'' depends on $s$ and $n$, but \citet{kandasamy15vonMises} suggest nonparametric estimators should only be used with $D$ at most $4$-$6$. \citet{rey12metaGaussianIB} tried using several nonparametric information estimators on the \emph{Communities and Crime} UCI data set ($n = 2195, D = 10$), but found all too unstable to be useful.}.

The second area is in the setting where data are assumed to be drawn from a truly Gaussian distribution. Here the high-dimensional case is far more optimistic. While this case had been studied previously~\citep{ahmed89entropy,misra05estimation,srivastava08bayesian}, \citet{cai15logDetCov} recently provided a precise finite-sample analysis based on deriving the exact probability law of the log-determinant $\log|\hat\Sigma|$ of the scatter matrix $\hat\Sigma$. From this, they derived a deterministic bias correction, giving an estimator for which they prove an MSE upper bound of $2\log \left( 1 - \frac{D}{n} \right)$ and a high-dimensional central limit theorem for the case $D \to \infty$ as $n \to \infty$ (but $D < n$).

\citet{cai15logDetCov} also prove a minimax lower bound of $2D/n$ on MSE, with several interesting consequences. First, consistent information estimation is possible only if $D/n \to 0$. Second, since, for small $x$, $\log(1 - x) \approx x$, this lower bound essentially matches the above upper bound when $D/n$ is small. Third, they show this lower bound holds even when restricted to diagonal covariance matrices. Since the upper bound for the general case and the lower bound for the diagonal case essentially match, it follows that Gaussian information estimation is not made easier by structural assumptions such as $\Sigma$ being bandable, sparse, or Toeplitz, as is common in, for example, stationary Gaussian process models~\citep{cai12adaptiveCovarianceMatrix}.

This $2D/n$ lower bound extends to our more general nonparanormal setting. However, we provide a minimax lower bound suggesting that the nonparanormal setting is strictly harder, in that optimal rates depend on $\Sigma$. Our results imply nonparanormal information estimation \emph{does} become easier if $\Sigma$ is assumed to be bandable or Toeplitz.

A closely related point is that known convergence rates for the fully nonparametric case require the density $p$ to be bounded away from $0$ or have particular tail behavior, due to singularity of the logarithm near $0$ and resulting sensitivity of Shannon information-theoretic functionals to regions of low but non-zero probability. In contrast, \citet{cai15logDetCov} need no lower-bound-type assumptions in the Gaussian case. In the nonparanormal case, we show \emph{some} such condition is needed to prove a uniform rate, but a weaker condition, a positive lower bound on $\lambda_D(\Sigma)$, suffices.


The {\bf main contributions} of this paper are the following:
\begin{enumerate}[noitemsep,topsep=0mm]
\item
We propose three estimators, $\hat I_G$, $\hat I_\rho$, and $\hat I_\tau$,\footnote{\citet{ince16I_G} proposed $\hat I_G$ for use in neuroimaging data analysis. To the best of our knowledge, $\hat I_\rho$ and $\hat I_\tau$ are novel.} for the mutual information of a nonparanormal distribution.
\item
We prove upper bounds, of order $O(D^2/(\lambda_D^2(\Sigma)n))$ on the mean squared error of $\hat I_\rho$, providing the first upper bounds for a nonparanormal information estimator. This bound suggests nonparanormal estimators scale far better with $D$ than nonparametric estimators.
\item
We prove a minimax lower bound suggesting that, unlike the Gaussian case, difficulty of nonparanormal information estimation depends on the true $\Sigma$.
\item
We give simulations comparing our proposed estimators to Gaussian and nonparametric estimators. Besides confirming and augmenting our theoretical predictions, these help characterize the settings in which each nonparanormal estimator works best.
\item
We present entropy estimators based on $\hat I_G$, $\hat I_\rho$, and $\hat I_\tau$. Though nonparanormal entropy estimation requires somewhat different assumptions from mutual information estimation, we show that entropy can also be estimated at the rate $O(D^2/(\lambda_D^2(\Sigma)n))$.

\end{enumerate}



\section{Nonparanormal Information Estimators}
\label{sec:estimators}

In this section, we present three different estimators, $I_G$, $I_\rho$, and $I_\tau$, for the mutual information of a nonparanormal distribution. We begin with a lemma providing common motivation for all three estimators.

Since mutual information is invariant to diffeomorphisms of individual variables, it is easy to see that the mutual information of a nonparanormal random variable is the same as that of the latent Gaussian random variable. Specifically:
\begin{lemma}
{\bf (Nonparanormal mutual information):} Suppose $X \sim \mathcal{NPN}(\Sigma; f)$. Then,
\begin{equation}
I(X) = -\frac{1}{2} \log|\Sigma|.
\label{eq:gaussian_MI}
\end{equation}
\label{lemma:NPN_MI}
\end{lemma}

Lemma~\ref{lemma:NPN_MI} shows that mutual information of a nonparanormal random variable depends only the latent covariance $\Sigma$; the marginal transformations are nuisance parameters, allowing us to avoid difficult nonparametric estimation; the estimators we propose all plug different estimates of $\Sigma$ into Eq.~\eqref{eq:gaussian_MI}, after a regularization step described in Section \ref{subsec:regularization}.

\subsection{Estimating $\Sigma$ by Gaussianization}
\label{subsec:I_G}

The first estimator $\hat \Sigma_G$ of $\Sigma$ proceeds in two steps. First, the data are transformed to have approximately standard normal marginal distributions, a process \citet{szabo07post} referred to as ``Gaussianization''. By the nonparanormal assumption, the Gaussianized data are approximately jointly Gaussian. Then, the latent covariance matrix is estimated by the empirical covariance of the Gaussianized data.

More specifically, letting $\Phi\inv$ denote the quantile function of the standard normal distribution and recalling the rank matrix $R$ defined in \eqref{def:empirical_rank_matrix}, the Gaussianized data
\[\tilde{X}_{i,j}
  := \Phi\inv \left( \frac{R_{i,j}}{n + 1} \right)
  \quad (\text{for } i \in [n], j \in [D])\]
are obtained by transforming the empirical CDF of the each dimension to approximate $\Phi$. Then, we estimate $\Sigma$ by the empirical covariance
$\hat \Sigma_G
  := \frac{1}{n}
      \sum_{i = 1}^n \tilde{X}_i \tilde{X}_i^T$.

\subsection{Estimating $\Sigma$ by rank correlation}
\label{subsec:I_rho_and_I_tau}
The second estimator actually has two variants, $I_\rho$ and $I_\tau$, respectively based on relating the latent covariance to two classic rank-based dependence measures, Spearman's $\rho$ and Kendall's $\tau$. For two random variables $X$ and $Y$ with CDFs $F_X,F_Y : \R \to [0, 1]$, $\rho$ and $\tau$ are defined by
\begin{align*}
\rho(X, Y) & := \Corr(F_X(X),F_Y(Y)) \\
\text{and } \quad
\tau(X, Y) & := \Corr(\sign(X - X'), \sign(Y - Y')),
\end{align*}
respectively, where \[\Corr(X, Y) = \frac{\E[(X - \E[X])(Y - \E[Y])]}{\sqrt{\Var[X]\Var[Y]}}\] denotes the standard Pearson correlation operator and $(X',Y')$ is an IID copy of $(X,Y)$. $\rho$ and $\tau$ generalize to the $D$-dimensional setting in the form of rank correlation matrices $\rho, \tau \in [-1,1]^{D \times D}$ with $\rho_{i,j} = \rho(X_i, X_j)$ and $\tau_{i,j} = \tau(X_i, X_j)$ for each $i \in [n],j \in [D]$.

$I_\rho$ and $I_\tau$ are based on a classical result relating the correlation and rank-correlation of a bivariate Gaussian:
\begin{theorem}
{\bf \citep{kruskal58ordinal}:} Suppose $(X,Y)$ has a Gaussian joint distribution with covariance $\Sigma$. Then,
\[\Corr(X, Y)
 = 2\sin \left(\frac{\pi}{6} \rho(X, Y) \right)
 = \sin \left( \frac{\pi}{2} \tau(X, Y) \right).\]
 \label{thm:kruskal}
\end{theorem}

$\rho$ and $\tau$ are often preferred to Pearson correlation for their relative robustness to outliers and applicability to non-numerical ordinal data. While these are strengths here as well, the main reason for their relevance is that they are invariant to marginal transformations (i.e., for diffeomorphisms $f, g : \R \to \R$, $\rho(f(X),g(Y)) = \pm \rho(X, Y)$ and $\tau(f(X),g(Y)) = \pm \tau(X,Y)$). As a consequence, the identity provided in Theorem \ref{thm:kruskal} extends unchanged to the case $(X,Y) \sim \mathcal{NPN}(\Sigma;f)$. This suggests an estimate for $\Sigma$ based on estimating $\rho$ or $\tau$ and plugging this element-wise into the transform $x \mapsto 2\sin \left( \frac{\pi}{6} x \right)$ or $x \mapsto \sin \left( \frac{\pi}{2} x \right)$, respectively. Specifically, $\Sigma_\rho$ is defined by
\[\hat\Sigma_\rho := 2 \sin\left( \frac{\pi}{6} \hat \rho \right),
  \quad \text{ where } \quad
  \hat\rho = \widehat{\Corr}(R)\]
is the empirical correlation of the rank matrix $R$, and sine is applied element-wise. Similarly, $\hat\Sigma_\tau := \sin\left( \frac{\pi}{2} \hat \tau \right)$,
where
\[\hat \tau_{j,k} := \frac{1}{\binom{n}{2}}\sum_{i \neq \ell \in [n]} \sign(X_{i,j} - X_{\ell,j})\sign(X_{i,k} - X_{\ell,k}).\]

\subsection{Regularization and estimating $I$}
\label{subsec:regularization}
Unfortunately, unlike usual empirical correlation matrices, none of $\hat\Sigma_G$, $\hat\Sigma_\rho$, or $\hat\Sigma_\tau$ is almost surely strictly positive definite. As a result, directly plugging into the mutual information functional~\eqref{eq:gaussian_MI} may give $\infty$ or even be undefined.
To correct for this, we propose a regularization step, in which we project each estimated latent covariance matrix onto the (closed) cone $\mathcal{S}(z)$ of symmetric matrices with minimum eigenvalue $z > 0$. Specifically, for any $z > 0$, let
\[\mathcal{S}(z) := \left\{ A \in \R^{D \times D} : A = A^T, \lambda_D(A) \geq z \right\}.\]
For any symmetric matrix $A \in \R^{D \times D}$ with eigendecomposition $\hat \Sigma = Q \Lambda Q\inv$ (i.e., $QQ^T = Q^TQ = I_D$ and $\Lambda$ is diagonal), the projection $A_z$ of $A$ onto $\mathcal{S}(z)$ is defined as $A_z := Q \Lambda_z Q\inv$, where $\Lambda_z$ is the diagonal matrix with $j^{th}$ nonzero entry $\left( \Lambda_z \right)_{j,j} = \max\{ z, \Lambda_{j,j} \}$. We call this a ``projection'' because $A_z$ is precisely the Frobenius norm projection of $A$ onto $\mathcal{S}(z)$ (see, e.g., \citet{henrion12semidefiniteProjection}):
$A_z = {\arg\!\min}_{B \in \R^{D \times D}} \|A - B\|_F$.

Applying this regularization to $\hat\Sigma_G$, $\hat\Sigma_\rho$, or $\hat\Sigma_\tau$ gives a strictly positive definite estimate $\hat\Sigma_{G,z}$, $\hat\Sigma_{\rho,z}$, or $\hat\Sigma_{\tau,z}$, respectively, of $\Sigma$. We can then estimate $I$ by plugging this into Equation \eqref{eq:gaussian_MI}, giving our three estimators:
\begin{align*}
\hat I_{G,z}
:= -\frac{1}{2} \log \left| \hat\Sigma_{G,z} \right|,
\qquad
\hat I_{\rho,z}
:= -\frac{1}{2} \log \left| \hat\Sigma_{\rho,z} \right| \\
\text{ and } \quad
\hat I_{\tau,z}
:= -\frac{1}{2} \log \left| \hat\Sigma_{\tau,z} \right|.
\end{align*}

\section{Upper Bounds on the Error of $\hat I_{\rho,z}$}
\label{sec:main_results}
%
%
Here, we provide finite-sample upper bounds on the error of the estimator $\hat I_\rho$ based on Spearman's $\rho$. Proofs are given in the Appendix.
We first bound the bias of the estimator:
\begin{proposition}
Suppose $X_1,...,X_n \stackrel{i.i.d.}{\sim} \mathcal{NPN}(\Sigma;f)$. Then, there exists a constant $C > 0$ such that, for any $z > 0$, the bias of $\hat I_{\rho,z}$ is at most
\begin{align*}
\left| \E \left[ \hat I_{\rho,z} \right] - I \right|
\leq C \left( \frac{D}{z\sqrt{n}} + \log \frac{|\Sigma_z|}{|\Sigma|} \right),
\end{align*}
where $\Sigma_z$ is the projection of $\Sigma$ onto $\mathcal{S}(z)$.
\label{prop:I_rho_bias_bound}
\end{proposition}

The first term of the bias stems from nonlinearity of the log-determinant function in Equation~\ref{eq:gaussian_MI}, which we analyze via Taylor expansion. The second term,
\[\log \frac{|\Sigma_z|}{|\Sigma|}
  = \sum_{\lambda_j(\Sigma) < z} \log \left( \frac{z}{\lambda_j(\Sigma)} \right),\]
is due to the regularization step and is actually, but is difficult to simplify or bound without further assumptions on the spectrum of $\Sigma$ and a choice of $z$, which we discuss later. We now turn to bounding the variance of $\hat I_{\rho,z}$. We first provide an exponential concentration inequality for $\hat I_{\rho,z}$ around its expectation, based on McDiarmid's inequality:
\begin{proposition}
Suppose $X_1,...,X_n \stackrel{i.i.d.}{\sim} \mathcal{NPN}(\Sigma;f)$. Then, for any $z,\e > 0$,
\[\pr \left[
    \left|
      \hat I_{\rho,z} - \E \left[ \hat I_{\rho,z} \right]
    \right| > \e
  \right]
  = 2 \exp \left(
    - \frac{n z^2\e^2}{18 \pi^2 D^2}
  \right).\]
\label{prop:I_rho_exponential_concentration}
\end{proposition}

Such exponential concentration bounds are useful when one wants to simultaneously bound the error of multiple uses of an estimator, and hence we present it separately as it may be independently useful. However, for the purpose of understanding convergence rates, we are more interested in the variance bound that follows as an easy corollary:

\begin{corollary}
Suppose $X_1,...,X_n \stackrel{i.i.d.}{\sim} \mathcal{NPN}(\Sigma;f)$. Then, for any $z > 0$, the variance of $\hat I_{\rho,z}$ is at most
\[\Var \left[ \hat I_{\rho,z} \right]
  \leq \frac{36 \pi^2 D^2}{z^2 n}.\]
\label{prop:I_rho_variance_bound}
\end{corollary}

Given these bias and variance bounds, a bound on the MSE of $\hat I_{\rho,z}$ follows via the usual bias-variance decomposition:
\begin{theorem}
Suppose $X \sim \mathcal{NPN}(\Sigma;f)$. Then, there exists a constant $C$ such that
\begin{align}
\E \left[ \left( \hat I_{\rho,z} - I \right)^2 \right]
\leq C \left( \frac{D^2}{z^2n} + \log^2 \frac{|\Sigma_z|}{|\Sigma|} \right).
\label{ineq:general_I_rho_MSE_bound}
\end{align}
\label{thm:general_I_rho_MSE_bound}
\end{theorem}
A natural question is now how to optimally select the regularization parameter $z$. While the bound \eqref{ineq:general_I_rho_MSE_bound} is clearly convex in $z$, it depends crucially on the unknown spectrum of $\Sigma$, and, in particular, on the smallest eigenvalues of $\Sigma$. As a result, it is difficult to choose $z$ optimally in general, but we we can do so for certain common subclasses of covariance matrices. For example, if $\Sigma$ is Toeplitz or bandable (i.e., for some $c \in (0,1)$, all $|\Sigma_{i,j}| \leq c^{|i - j|}$), then the smallest eigenvalue of $\Sigma$ can be bounded below~\citep{cai12adaptiveCovarianceMatrix}. When $\Sigma$ is bandable, as we show in the Appendix, this bound can be independent of $D$. In these cases, the following somewhat simpler MSE bound can be used:
\begin{corollary}
Suppose $X \sim \mathcal{NPN}(\Sigma;f)$, and suppose $z \leq \lambda_D(\Sigma)$. Then, there exists a constant $C > 0$ such that
\[\E \left[ \left( \hat I_{\rho,z} - I \right)^2 \right] \leq \frac{CD^2}{z^2n}.\]
\label{corr:specific_I_rho_MSE_bound}
\end{corollary}

\section{Lower Bounds in terms of $\Sigma$}
\label{sec:Sigma_lower_bound}
When the data $X_1,...,X_n \stackrel{i.i.d}{\sim} \mathcal{N}(0,\Sigma)$ are truly Gaussian, using the plug-in estimator
\[\textstyle \hat I = -\frac{1}{2} \log \left| \hat \Sigma \right|
\quad \mbox{ (where } \quad
\hat\Sigma = \frac{1}{n} \sum_{i = 1}^n X_i X_i^T\]
is the empirical covariance matrix), \citet{cai15logDetCov} showed that the distribution of $\hat I - I$ is independent of the true correlation matrix $\Sigma$. This follows from  the ``stability'' of Gaussians (i.e., that nonsingular linear transformations of Gaussian random variables are Gaussian). In particular,
\[\hat I - I = \log|\hat\Sigma| - \log|\Sigma| = \log|\Sigma^{-1/2}\hat\Sigma\Sigma^{-1/2}|,\] and $\Sigma^{-1/2}\hat\Sigma\Sigma^{-1/2}$ has the same distribution as $\log\hat\Sigma$ does in the special case that $\Sigma = I_D$ is the identity. This property is both somewhat surprising, given that $I \to \infty$ as $|\Sigma| \to 0$, and useful, leading to a tight analysis of the error of $\hat I$ and confidence intervals that do not depend on $\Sigma$. 

It would be convenient if any nonparanormal information estimators satisfied this property. Unfortunately, the main result of this section is a negative one, showing that this property is unlikely to hold without additional assumptions:

\begin{proposition}
Consider the $2$-dimensional case
\begin{equation}
X_1,...,X_n \stackrel{i.i.d}{\sim} \mathcal{N}(0,\Sigma), \quad \text{ with } \quad
\Sigma = \begin{bmatrix}
1 & \sigma \\
\sigma & 1
\end{bmatrix},
\label{eq:2D_Sigma}
\end{equation}
and let $\sigma_* \in (0,1)$. Suppose an estimator $\hat I = \hat I(R)$ of $I_\sigma = -\frac{1}{2}\log(1 - \sigma^2)$ is a function of the empirical rank matrix $R \in \N^{n \times 2}$ of $X$. Then, there exists a constant $C > 0$, depending only $n$, such that the worst-case MSE of $\hat I$ over $\sigma \in (0,\sigma_*)$ satisfies
\begin{align*}
\sup_{\sigma \in (0,\sigma^*)} \E \left[ \left( \hat I(R) - I_\sigma \right)^2 \right]
& \geq \frac{1}{64} \left( C - \log(1 - \sigma_*^2) \right)^2
\end{align*}
\label{prop:sigma_lower_bound}
\end{proposition}
Clearly, this lower bound tends to $\infty$ as $\sigma \to 1$.
As written, this result lower bounds the error of \emph{rank-based estimators} in the Gaussian case when $\sigma \approx 1$.
However, to the best of our knowledge, all methods for estimating $\Sigma$ in the nonparanormal case are functions of $R$, and prior work~\citep{hoff07extending} has shown that the rank matrix $R$ is a generalized sufficient statistic for $\Sigma$ (and hence for $I$) in the nonparanormal model. Thus, it is reasonable to think of lower bounds for rank-based estimators in the Gaussian case as lower bounds for any estimator in the nonparanormal case.

The proof of this result is based on the simple observation that the rank matrix can take only finitely many values. Hence, as $\sigma \to 1$, $R$ tends to be perfectly correlated, providing little information about $\sigma$, whereas the dependence of the estimand $I_\sigma$ on $\sigma$ increases sharply. This is intuition is formalized in the Appendix
using Le Cam's lemma for lower bounds in two-point parameter estimation problems.

\section{Empirical Results}
\label{sec:empirical}
We compare 5 mutual information estimators:
\begin{itemize}[itemsep=0.0mm,topsep=0mm]
\item
$\hat I$: Gaussian plug-in estimator with bias-correction (see \citet{cai15logDetCov}).
\item
$\hat I_G$: Nonparanormal estimator using Gaussianization.
\item
$\hat I_\rho$: Nonparanormal estimator using Spearman's $\rho$.
\item
$\hat I_\tau$: Nonparanormal estimator using Kendall's $\tau$.
\item
$\hat I_{k\text{NN}}$: Nonparametric estimator using $k$-nearest neighbor ($k$NN) statistics.
\end{itemize}

For $I_\rho$ and $I_\tau$, we used a regularization constant $z = 10^{-3}$. We did not regularize for $I_G$. Although this implies $\pr[I_G~=~\infty]~>~0$, this is extremely unlikely for even moderate values of $n$ and never occurred during our experiment, which all use $n \geq 32$. We will thus omit denoting dependence on $z$. For $I_{k\text{NN}}$, except as noted in Experiment 3, $k = 2$, based on recent analysis \citep{singh16kNNEntropy} suggesting that small values of $k$ are best for estimation.

Sufficient details to reproduce experiments are given in the Appendix,
and MATLAB source code is available at [Omitted for anonymity]. We report MSE based on $1000$ i.i.d. trials of each condition. $95\%$ confidence intervals were consistently smaller than plot markers and hence omitted to avoid cluttering plots. Except as specified otherwise, each experiment had the following basic structure: In each trial, a correlation matrix $\Sigma$ was drawn by normalizing a random covariance matrix from a Wishart distribution, and data $X_1,...,X_n \stackrel{i.i.d.}{\sim} \mathcal{N}(0, \Sigma)$ drawn. All $5$ estimators were computed from $X_1,...,X_n$ and squared error from true mutual information (computed from $\Sigma$) was recorded. Unless specified otherwise, $n = 100$ and $D = 25$.

Since our nonparanormal information estimators are functions of ranks of the data, neither the true mutual information nor our non-paranormal estimators depend on the marginal transformations. Thus, except in Experiment 2, where we show the effects of transforming marginals, and Experiment 3, where we add outliers to the data, we perform all experiments on truly Gaussian data, with the understanding that this setting favors the Gaussian estimator.

All experimental results are displayed in Figure~\ref{fig:main_experimental_results}.

\begin{figure*}[t!]
    \centering
    \begin{subfigure}[b]{0.27\textwidth}
        \centering
        \includegraphics[width=\textwidth,trim={0 0 0 0},clip]{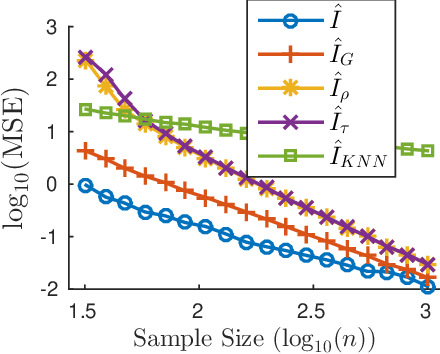}
        \caption{Experiment 1}
        \label{subfig:exp_1}
    \end{subfigure}%
    ~
    \begin{subfigure}[b]{0.24\textwidth}
        \centering
        \includegraphics[width=\textwidth,trim={0mm 0 0 0},clip]{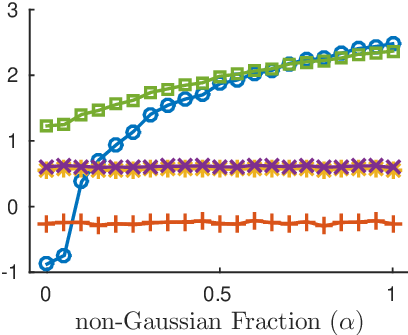}
        \caption{Experiment 2}
        \label{subfig:exp_2}
    \end{subfigure}%
    ~
    \begin{subfigure}[b]{0.24\textwidth}
        \centering
        \includegraphics[width=\textwidth,trim={0mm 0 0 0},clip]{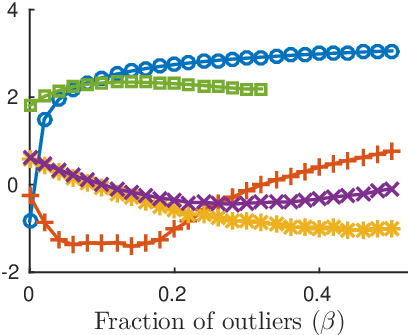}
        \caption{Experiment 3}
        \label{subfig:exp_3}
    \end{subfigure}%
    ~
    \begin{subfigure}[b]{0.25\textwidth}
        \centering
        \includegraphics[width=\textwidth,trim={1mm 0 0 0},clip]{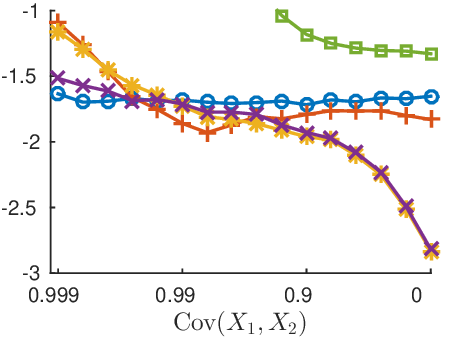}
        \caption{Experiment 4}
        \label{subfig:exp_4}
    \end{subfigure}
    \caption{Plots of $\log_{10}(\text{MSE})$ plotted over (a) log-sample-size $\log_{10}(n)$, (b) fraction $\alpha$ of dimensions with non-Gaussian marginals, (c) fraction $\beta$ of outlier samples in each dimension, and (d) covariance $\Sigma_{1,2} = \Cov(X_1,X_2)$. Note that the $x$-axis in (d) is decreasing.}
\label{fig:main_experimental_results}
\end{figure*}


{\bf Experiment 1 (Dependence on $n$):}
We first show nonparanormal estimators have ``parametric'' $O(n\inv)$ dependence on $n$, unlike $\hat I_{k\text{NN}}$, which converges far more slowly. For large $n$, MSEs of $\hat I_G$, $\hat I_\rho$, and $\hat I_\tau$ are close to that of $\hat I$.

{\bf Experiment 2 (Non-Gaussian Marginals):}
\label{subsec:hat_I_nonrobust}
Next, we show nonparanormal estimators are robust to non-Gaussianity of the marginals, unlike $\hat I$. We applied a nonlinear transformation $f$ to a fraction $\alpha \in [0, 1]$ of dimensions of Gaussian data. That is, we drew $Z_1,...,Z_n \stackrel{i.i.d.}{\sim} \mathcal{N}(0,\Sigma)$ and then used data $X_1,...,X_n$, where
\[X_{i,j} = \left\{
\begin{array}{ll}
T(Z_{i,j}) & \mbox{ if } j < \alpha D \\
Z_{i,j} & \mbox{ if } j \geq \alpha D
\end{array}
\right., \quad \forall i \in [n], j \in [D],\]
for a diffeomorphism $T$. Here, we use $T(z) = e^z$. The Appendix shows similar results for several other $T$. $\hat I$ performs poorly even when $\alpha$ is quite small. Poor performance of $\hat I_{k\text{NN}}$ may be due to discontinuity of the density at $x = 0$.

{\bf Experiment 3 (Outliers):}
We now show that nonparanormal estimators are far more robust to the presence of outliers than $\hat I$ or $\hat I_{k\text{NN}}$. To do this, we added outliers to the data according to the method of \citet{liu12SKEPTIC}. After drawing Gaussian data, we independently select $\lfloor \beta n \rfloor$ samples in each dimension, and replace each i.i.d. uniformly at random from $\{-5,+5\}$. Performance of $\hat I$ degrades rapidly even for small $\beta$. $\hat I_{k\text{NN}}$ can fail for atomic distributions, $\hat I_{k\text{NN}} = \infty$ whenever at least $k$ samples are identical. This mitigate this, we increased $k$ to $20$ and ignored trials where $\hat I_{k\text{NN}} = \infty$, but $\hat I_{k\text{NN}}$ ceased to give any finite estimates when $\beta$ was sufficiently large.

For small values of $\beta$, nonparanormal estimators surprisingly improve. We hypothesize this is due to convexity of the mutual information functional Eq.~\eqref{eq:gaussian_MI} in $\Sigma$. By Jensen's inequality, estimators which plug-in an approximately unbiased estimate $\hat\Sigma$ of $\Sigma$ are biased towards overestimating $I$. Adding random (uncorrelated) noise reduces estimated dependence, moving the estimate closer to the true value.
If this nonlinearity is indeed a major source of bias, it may be possible to derive a von Mises-type bias correction (see \citet{kandasamy15vonMises}) accounting for higher-order terms in the Taylor expansion of the log-determinant.


{\bf Experiment 4 (Dependence on $\Sigma$):}
Here, we verify our results in \hyperref[sec:Sigma_lower_bound]{Section~\ref{sec:Sigma_lower_bound}} showing that MSE of rank-based estimators approaches $\infty$ as $|\Sigma| \to 0$, while MSE of $\hat I$ is independent of $\Sigma$. Here, we set $D = 2$ and $\Sigma$ as in Eq.~\eqref{eq:2D_Sigma}, varying $\sigma \in [0,1]$. Indeed, the MSE of $\hat I$ does not change, while the MSEs of $\hat I_G$, $\hat I_\rho$, and $\hat I_\tau$ all increase as $\sigma \to 1$. This increase seems mild in practice, with performance worse than of $\hat I$ only when $\sigma > 0.99$. $\hat I_\tau$ appears to perform far better than $\hat I_G$ and $\hat I_\rho$ in this regime.
Performance of $I_{k\text{NN}}$ degrades far more quickly as $\sigma \to 1$. This phenomenon is explored by \citet{gao15efficient}, who lower bound error of $I_{k\text{NN}}$ in the presence of strong dependencies, and proposed a correction to improve performance in this case.

It is also interesting that errors of $\hat I_{\rho}$ and $\hat I_{\tau}$ drop as $\sigma \to 0$. This is likely because, in this regime, the main source of error is the variance of $\hat\rho$ and $\hat\tau$ (as $-\log(1 - \sigma^2) \approx \sigma^2$ when $\sigma \approx 0$). When $n \to \infty$ and $D$ is fixed, both $2\sin(\pi\hat\rho/6)$ and $\sin(\pi\hat\tau/2)$ are asymptotically normal estimates of $\sigma$, with asymptotic variances proportional to $(1 - \sigma^2)^2$~\citep{klaassen97bivariateNormalCopula}. By the delta method, since $\frac{dI}{d\sigma} = \frac{\sigma}{1 - \sigma^2}$, $\hat I_\rho$ and $\hat I_\tau$ are asymptotically normal estimates of $I$, with asymptotic variances proportional to $\sigma^2$ and hence vanishing as $\sigma \to 0$.

\section{Estimating Entropy}
\label{sec:entropy}
Thus far, we have discussed estimation of mutual information $I(X)$. Mutual information is convenient because it is invariant under marginal transformation, and hence $I(X) = I(f(X))$ depends only on $\Sigma$. While the entropy $H(X)$ does depend on the marginal transform $f$, fortunately, by Eq.~\eqref{eq:entropy_relation}, $H(X)$ differs from $I(X)$ only by a sum of univariate entropies. Univariate nonparametric estimation of entropy in has been studied extensively, and there exist several estimators (e.g., based on sample spacings~\citep{beirlant97entropyOverview}, kernel density estimates~\citep{moon16improving} or $k$-nearest neighbor methods~\citep{singh16kNNEntropy}) that can estimate $H(X_j)$ at the rate $\asymp n\inv$ in MSE under relatively mild conditions on the marginal density $p_j$. While the precise assumptions vary with the choice of estimator, they are mainly (a) that $p_j$ be lower bounded on its support or have particular (e.g., exponential) tail behavior, and (b) that $p_j$ be smooth, typically quantified by a H\"older or Sobolev condition. Details of these assumptions are in the Appendix. 

Under these conditions, since there exist estimators $\hat H_1,...,\hat H_D$ and a constant $C > 0$ such that
\begin{equation}
\E [(\hat H_j - H(X_j))^2] \leq C/n,
\quad \forall j \in [D].
\label{ineq:MSE_of_sum_of_marginal_entropies}
\end{equation}
Combining these estimators with an estimator, say $\hat I_{\rho,z}$, of mutual information gives an estimator of entropy:
\[\textstyle \hat H_{\rho,z} := \sum_{j = 1}^D \hat H_j - \hat I_{\rho,z}.\]
If we assume $z = \lambda_D\inv(\Sigma)$ is bounded below by a positive constant, combining inequality~\eqref{ineq:MSE_of_sum_of_marginal_entropies} with Corollary~\ref{corr:specific_I_rho_MSE_bound} gives
\[\E \left[ \left( \hat H_{\rho,z} - H(X) \right)^2 \right] \leq \frac{CD^2}{n},\]
where the constant $C$ may differ from in \eqref{ineq:MSE_of_sum_of_marginal_entropies} but is independent of $n$ and $D$.


\section{Conclusions and Future Work}
\label{sec:conc_and_future}

This paper we suggests nonparanormal information estimation as a practical compromise between the difficult nonparametric case and the restrictive Gaussian case. We proposed three estimators for this problem, and provided the first upper bounds for nonparanormal information estimation. We also provided lower bounds showing how dependence on $\Sigma$ differs from the Gaussian case, and we demonstrated empirically that nonparanormal estimators are more robust than Gaussian estimators, even when dimension is too high for fully nonparametric estimators.

Collectively, these results suggest that, by scaling to moderate or high dimensionality without relying on Gaussianity, nonparanormal information estimators may be effective tools with a number of machine learning applications. While the best choice of information estimator inevitably depends on context, as a crude off-the-shelf guide for practitioners, the estimators we might suggest, in order of preference, are:
\begin{itemize}[leftmargin=*,noitemsep,topsep=0pt]
\item
fully nonparametric if $D < 6, n > \max\{100,10^D\}$.
\item
$\hat I_\rho$ if $D^2/n$ is small and data may have outliers.
\item
$\hat I_\tau$ if $D^2/n$ is small and dependencies may be strong.
\item
$\hat I_G$ otherwise.
\item
$\hat I$ only given strong belief that data are nearly Gaussian.
\end{itemize}


There are many natural open questions in this line of work. First, in the nonparanormal model, we focused on estimating mutual information $I(X)$, which does not depend on marginal transforms $f$, and entropy, which decomposes into $I(X)$ and $1$-dimensional entropies. In both cases, additional structure imposed by the nonparanormal model allows estimation in higher dimensions than fully nonparametric models. Can nonparanormal assumptions lead to higher dimensional estimators for the many other useful nonlinear functionals of densities (e.g., $L_p$ norms/distances and more general (e.g., R\'enyi or Tsallis) entropies, mutual informations, and divergences) that do not decompose?

Second, there is a gap between our upper bound rate of $\|\Sigma\inv\|_2^2 D^2/n$ and the only known lower bound of $2D/n$ (from the Gaussian case), thought we also showed that bounds for rank-based estimators depend on $\Sigma$. Is quadratic dependence on $D$ optimal? How much do rates improve under structural assumptions on $\Sigma$? Upper bounds should be derived for other estimators, such as $\hat I_G$ and $\hat I_\tau$. The $2D/n$ lower bound proof of \citet{cai15logDetCov} for the Gaussian case, based on the Cramer-Rao inequality~\citep{van07parameter}, is unlikely to tighten in the nonparanormal case, since Fisher information is invariant to diffeomorphisms of the data. Hence, a new approach is needed if the lower bound in the nonparanormal case is to be raised.



Finally, our work also applies to estimating the log-determinant $\log|\Sigma|$ of the latent correlation matrix in a nonparanormal model. In addition to information estimation, the work of \citet{cai15logDetCov} on estimating $\log|\Sigma|$ in the Gaussian setting was motivated by the use of $\log|\Sigma|$ in several other multivariate statistical tools, including
quadratic discriminant analysis (QDA) and
MANOVA~\citep{anderson84multivariate}. Can our estimators lead to more robust nonparanormal versions of these procedures?

\bibliography{biblio}
\bibliographystyle{icml2017}

\appendix
\label{appendix}

\section{Lemmas}
Our proofs rely on the following lemmas.
\begin{lemma}
{\bf (Convexity of the inverse operator norm):}
The function $A \mapsto \|A\inv\|_2$ is convex over $A \succ 0$.
\label{lemma:inv_op_norm_cvx}
\end{lemma}
\begin{proof}
For $A, B \succ 0$, let $C := \tau A + (1 - \tau) B$. Then,
\begin{align*}
\| \hat C\inv \|_2
& = \frac{1}{\inf_{x \in \R^D} x^T C x} \\
& = \frac{1}{\inf_{x \in \R^D} \tau x^T A x + (1 - \tau) x^T B x} \\
& \leq \frac{1}{\tau \inf_{x \in \R^D} x^T A x + (1 - \tau) \inf_{x \in \R^D} x^T B x} \\
& \leq \tau \frac{1}{\inf_{x \in \R^D} x^T A x}
  + (1 - \tau) \frac{1}{\inf_{x \in \R^D} x^T B x} \\
& = \tau \left\| A\inv \right\|_2 + (1 - \tau) \left\| B\inv \right\|_2
\end{align*}
via convexity of the function $x \mapsto 1/x$ on $(0, \infty)$.
\end{proof}

\begin{lemma}
{\bf (Mean-Value Bound on the Log-Determinant):}
Matrix derivative of log-determinant. Suppose $A, B \succ 0$. Then, for $\lambda := \min\{\lambda_D(A), \lambda_D(B)\}$,
\[\left| \log |A| - \log |B| \right| \leq \frac{1}{\lambda} \|A - B\|_F.\]
\label{lemma:log_det_mean_value}
\end{lemma}
\begin{proof}
\emph{Proof:} First recall that the log-determinant is continuously differentiable over the strict positive definite cone, with $\nabla_X \log |X| = X\inv$ for any $X \succ 0$. Hence, by the matrix-valued version of the mean value theorem,
\[\log |A| - \log |B| = tr(C\inv(A - B)),\]
where $C = \tau A + (1 - \tau) B$ for some $\tau \in (0, 1)$. Since for positive definite matrices, the inner product can be bounded by the product of the operator and Frobenius norms, and clearly $C \succ 0$, we have
\[\left| \log |A| - \log |B| \right| = \|C\inv\|_2\|A - B\|_F.\]
Finally, it follows by Lemma \ref{lemma:inv_op_norm_cvx} that
\[\left| \log |A| - \log |B| \right| \leq \frac{1}{\lambda} \|A - B\|_F.\]
\end{proof}

\section{Proofs of Main Results}

Here, we give proofs of our main theoretical results, beginning with upper bounds on the MSE of $\hat I_\rho$ and proceeding to minimax lower bounds in terms of $\Sigma$.

\section{Upper bounds on the MSE of $\hat I_\rho$}

\begin{proposition}
\begin{align*}
& \left| \E \left[ \log |\hat \Sigma_z| \right] - \log \left|\Sigma\right| \right| \\
& \leq C \left( \|\Sigma\|_2^2\frac{D}{z^2n} + \left( \sum_{\lambda_j(\Sigma) < z} \log \left( \frac{z}{\lambda_j(\Sigma)} \right) \right)^2 \right).
\end{align*}
\label{prop:I_rho_bias_bound_appendix}
\end{proposition}
\begin{proof}
By the triangle inequality,
\begin{align*}
\left| \E \left[ \log |\hat \Sigma_z| \right] - \log \left|\Sigma\right| \right|
& \leq \left| \E \left[ \log |\hat \Sigma_z| \right] - \log \left|\Sigma_z\right| \right| \\
&   + \left| \log |\Sigma_z| - \log |\Sigma| \right|
\end{align*}
For the first term, applying the matrix mean value theorem (Lemma~\ref{lemma:log_det_mean_value}) and the inequality
$\|A\|_F \leq \sqrt{D} \|A\|_2$
\begin{align*}
\left| \E \left[ \log \left| \hat \Sigma_z \right| \right] - \log |\Sigma_z| \right|
& \leq \E \left[ \left| \log \left| \hat \Sigma_z \right| - \log |\Sigma_z| \right| \right] \\
& \leq \frac{1}{z} \E \left[ \left\| \hat \Sigma_z - \Sigma_z \right\|_F \right] \\
& \leq \frac{\sqrt{D}}{z} \E \left[ \left\| \hat \Sigma_z - \Sigma_z \right\|_2 \right] \\
& \leq \frac{C_{MZ} \|\Sigma\|_2 D}{z\sqrt{n}},
\end{align*}
where we used Theorem 1 of \citet{mitra14nonparametricCovariance}, which gives a
constant $C_{MZ}$ such that
\[\E \left[ \left\| \hat \Sigma_z - \Sigma_z \right\|_2 \right] \leq C_{MZ} \|\Sigma\|_2 \sqrt{\frac{D}{n}}.\]
Via the bound $\|\Sigma\|_2 \leq \sqrt{D} \|\Sigma\|_\infty$, this reduces to
\[\E \left[ \left\| \hat \Sigma_z - \Sigma_z \right\|_2 \right]
  \leq C_{MZ} \frac{D}{\sqrt{n}}.\]

\end{proof}

\begin{proposition}
\[\Var \left[ \hat I \right]
  \leq \frac{36 \pi^2 D^2}{z^2 n}.\]
\label{prop:I_rho_variance_bound_appendix}
\end{proposition}

\begin{proof}
By the Efron-Stein inequality, since $X_1,\dots,X_n$ are independent and identically distributed,
\begin{align*}
\Var \left[ \hat I \right]
& \leq \frac{1}{2}
        \sum_{i = 1}^n
          \E \left[
            \left(
              \log |\hat \Sigma_z| - \log |\hat \Sigma_z^{(i)}|
            \right)^2
          \right] \\
& = \frac{n}{2}
          \E \left[
            \left(
              \log |\hat \Sigma_z| - \log |\hat \Sigma_z^{(1)}|
            \right)^2
          \right],
\end{align*}
where $\hat \Sigma_z^{(1)}$ is our estimator after independently re-sampling the
first sample $X_1$. Applying the multivariate mean-value theorem (Lemma \ref{lemma:log_det_mean_value}), we have
\[\left| \log |\hat \Sigma_z| - \log |\hat \Sigma_z^{(1)}| \right|
  \leq \frac{1}{z} \|\hat \Sigma_z - \hat \Sigma_z^{(1)} \|_F.\]
$\| \hat \Sigma_\tau\inv \|_2 \leq \frac{1}{z}$.
Since $\mathcal{S}(z)$ is convex and the Frobenius norm is supported by an inner product, the operation of projecting onto $\mathcal{S}(z)$ is a contraction. In particular,
$\left\| \left( \hat \Sigma_z - \hat \Sigma_z^{(1)} \right) \right\|_F
  \leq \left\| \left( \hat \Sigma - \hat \Sigma^{(1)} \right) \right\|_F$
Applying the mean value theorem to the function $x \mapsto 2\sin \left( \frac{\pi}{6} x \right)$,
\begin{align}
\left\| \left( \hat \Sigma - \hat \Sigma^{(1)} \right) \right\|_F^2
& = \sum_{j,k = 1}^D \left( \hat \Sigma - \hat \Sigma^{(1)} \right)_{j,k}^2 \\
& \leq \frac{\pi^2}{9} \sum_{j,k = 1}^D \left( \hat \rho_{j,k} - \hat \rho_{j,k}^{(1)} \right)^2 \\
& = \frac{\pi^2}{9} \left\| \hat \rho - \hat \rho^{(1)} \right\|_F^2.
\label{ineq:sine_MVT}
\end{align}
From the formula
\[\hat \rho_{j,k} = 1 - \frac{6 \sum_{i = 1}^n d_{i,j,k}^2}{n(n^2 - 1)},\]
(where $d_{i,j,k}$ denotes the difference in ranks of $X_{i,j}$ and $X_{i,k}$ in $X_{1,j},...,X_{n,j}$ and $X_{1,k},...,Y_{n,k}$, respectively), one can see, since $|d_{1,j,k} - d_{1,j,k}'| \leq n$ and, for $i \neq 1$, $|d_{i,j,k} - d_{i,j,k}'| \leq 1$, that
\[\left| \hat \rho_{j,k} - \hat \rho_{j,k}^{(1)} \right| \leq \frac{18}{n},\]
and hence that
\begin{equation}
\| \hat \rho - \hat \rho^{(1)} \|_F \leq \frac{18D}{n}.
\label{ineq:frobenius_rho}
\end{equation}
It follows from inequality \eqref{ineq:sine_MVT} that
\[\|\hat \Sigma_z - \hat \Sigma_z^{(1)}\|_F \leq \frac{6 \pi D}{n}.\]
Altogether, this gives
\[\left| \log |\hat \Sigma_z| - \log |\hat \Sigma_z^{(1)}| \right|
  \leq \frac{6 \pi D}{z n}.\]
Then, McDiarmid's Inequality gives, for all $\e > 0$,
\[\pr \left[ \left| \hat I - \E \left[ \hat I \right] \right| > \e \right]
  = 2 \exp \left(
    - \frac{n z^2\e^2}{18 \pi^2 D^2}
  \right).\]
This translates to a variance bound of
\[\Var \left[ \hat I \right]
  \leq \frac{36 \pi^2 D^2}{z^2 n}.\]
\end{proof}

\subsection{Lower bound for rank-based estimators in terms of $\Sigma$}
One (perhaps surprising) result of \citet{cai15logDetCov} is that, as long as $D/n \to 0$, the convergence rate of the estimator is independent of the true correlation structure $\Sigma$. Here, we show that this desirable property does not hold in the nonparanormal case.

\begin{proposition}
Consider the $2$-dimensional case
\begin{equation}
X_1,...,X_n \stackrel{i.i.d}{\sim} \mathcal{N}(0,\Sigma), \quad \text{ with } \quad
\Sigma = \begin{bmatrix}
1 & \sigma \\
\sigma & 1
\end{bmatrix},
\label{eq:2D_Sigma_appendix}
\end{equation}
and let $\sigma_* \in (0,1)$. Suppose an estimator $\hat I = \hat I(R)$ of $I_\sigma = -\frac{1}{2}\log(1 - \sigma^2)$ is a function of the empirical rank matrix $R \in \N^{n \times 2}$ of $X$ (as defined in \eqref{def:empirical_rank_matrix}). Then, there exists a constant $C > 0$, depending only $n$, such that the worst-case MSE of $\hat I$ over $\sigma \in (0,\sigma_*)$ satisfies
\begin{align*}
\sup_{\sigma \in (0,\sigma^*]} \E \left[ \left( \hat I(R) - I_\sigma \right)^2 \right]
& \geq \frac{1}{64} \left( C - \log(1 - \sigma_*^2) \right)^2 \\
& \to \infty \quad \text{ as } \quad \sigma_* \to 1.
\end{align*}
\label{prop:sigma_lower_bound_appendix}
\end{proposition}
\begin{proof}
Note that the rank matrix $R$ can take only finitely many values. Let $\mathcal{R}$ be the set of all $(n!)^D$ possible rank matrices and let $\mathcal{R}_1 \subseteq \mathcal{R}$ be the set of rank matrices that are perfectly correlated. Then, as $\sigma \to 1$, $\pr[R \in \mathcal{R}_1] \to 1$, so, in particular, we can pick $\sigma_0$ (depending only on $n$) such that, for all $\sigma \geq \sigma_0$, $\pr[R \in \mathcal{R}_1] \geq \frac{1}{2}$. Since the data are i.i.d., all rank matrices in $\mathcal{R}_1$ have equal probability. It follows that
\[D_{TV}(\pr_0||\pr_1)
 = \frac{1}{2}\|\pr_0 - \pr_1\|_1
 \leq \frac{1}{2}.\]
Finally, by Le Cam's Lemma (see, e.g., Section 2.3 of \citet{Tsybakov:2008:INE:1522486}),
\begin{align*}
& \inf_{\hat I} \sup_{\sigma \in \{\sigma_0,\sigma_1\}} \E \left[ \left( \hat I - I_\sigma \right)^2 \right] \\
& \geq \frac{(I_{\sigma_*} - I_{\sigma_0})^2}{8} \left( 1 - D_{TV}(P_{\sigma_0},P_{\sigma_1}) \right) \\
& \geq \frac{(\log(1 - \sigma_0^2) - \log(1 - \sigma_*^2))^2}{64} \\
\end{align*}
\end{proof}

\section{Details of Experimental Methods}
\label{sec:experiment_details}
Here, we present details needed to reproduce our numerical simulations. Note that MATLAB source code for these experiments is available at [Omitted for anonymity.], including a single runnable script that performs all experiments and generates all figures presented in this paper.
Specific details needed to reproduce experiments are given in the Appendix,

In short, experiments report empirical mean squared errors based on $100$ i.i.d. trials of each condition. We initially computed $95\%$ confidence intervals, but these intervals were consistently smaller than marker sizes, so we omitted them to avoid cluttering plots. Except as specified otherwise, each experiment followed the same basic structure, as follows: In each trial, a random correlation matrix $\Sigma \in [-1,1]^{D \times D}$ was drawn by normalizing a covariance matrix from a Wishart distribution $W(I_D,D)$ with identity scale matrix and $D$ degrees of freedom. Data $X_1,...,X_n$ were then drawn i.i.d. from $\mathcal{N}(0, \Sigma)$. All estimators were applied to the same data. Unless specified otherwise, $n = 100$ and $D = 25$.

\subsection{Computational Considerations}
In general, the running time of all the nonparanormal estimators considered is $O(D n \log n + D^2 n + D^3)$ (i.e., $O(Dn \log n)$ to rank or Gaussianize the variables in each dimension, $D^2n$ to compute the covariance matrix, and $O(D^3)$ to compute the log-determinant). All log-determinants $\log|\Sigma|$ were computed by summing the logarithms of the diagonal of the Cholesky decomposition of $\Sigma$, as this is widely considered to be a fast and numerically stable approach. Note however that faster ($O(D)$-time) randomized algorithms~\citep{han15largeLogDet} have been proposed to approximate the log-determinant).

\section{Additional Experimental Results}
\label{sec:additional_experiments}
Here, we present variants on the experiments presented in the main paper, which support but are not necessary for illustrating our conclusions.
\subsection{Effects of Other Marginal Transformations}
\label{subsec:other_marginal_transforms_appendix}
In \hyperref[subsec:hat_I_nonrobust]{Section~\ref{subsec:hat_I_nonrobust}}, we showed that the Gaussian estimator $\hat I$ is highly sensitive to failure of the Gaussian assumption for even a small fraction of marginals. Figure~\ref{subfig:exp_2}, illustrates this for the transformation $x \mapsto \exp(x)$, but we show here that this is not specific to the exponential transformation. As shown in Figures~\ref{fig:hat_I_nonrobust_appendix} nearly identical results hold when the marginal transformation $f$ is the hyperbolic tangent function $x \mapsto \tanh(x)$, the cubic function $x \mapsto x^3$, sigmoid function $x \mapsto \frac{1}{1 + e^{-x}}$, or standard normal CDF.

\begin{figure*}[t!]
    \centering
    \begin{subfigure}[b]{0.26\textwidth}
        \centering
        \includegraphics[width=\textwidth,trim={0 0 0 0},clip]{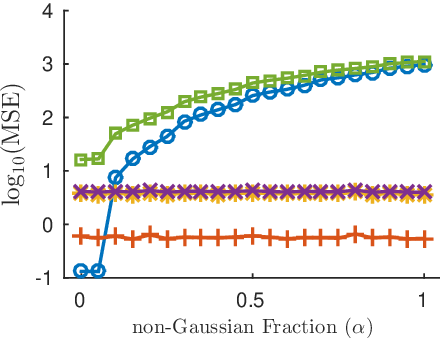}
        \caption{$T(x) = x^3$}
    \end{subfigure}%
    ~
    \begin{subfigure}[b]{0.24\textwidth}
        \centering
        \includegraphics[width=\textwidth,trim={5mm 0 0 0},clip]{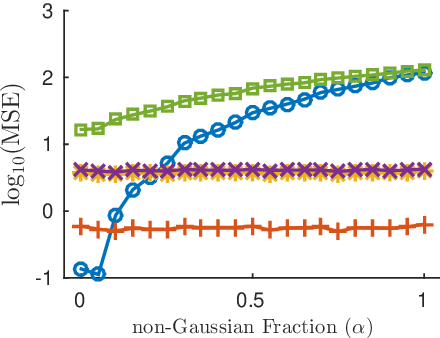}
        \caption{$T(x) = \tanh(x)$}
    \end{subfigure}%
    ~
    \begin{subfigure}[b]{0.24\textwidth}
        \centering
        \includegraphics[width=\textwidth,trim={5mm 0 0 0},clip]{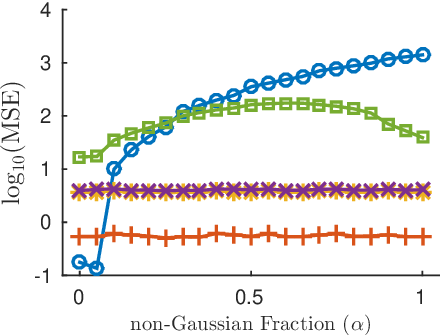}
        \caption{$T(x) = \frac{1}{1 + e^{-x}}$}
    \end{subfigure}%
    ~
    \begin{subfigure}[b]{0.24\textwidth}
        \centering
        \includegraphics[width=\textwidth,trim={5mm 0 0 0},clip]{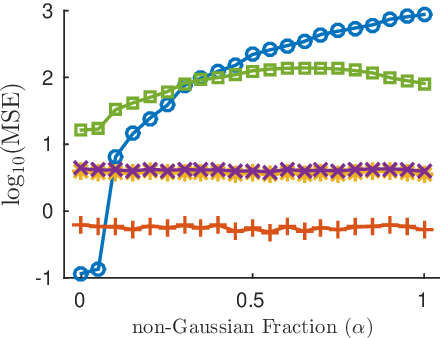}
        \caption{$T(x) = \Phi(x)$}
    \end{subfigure}
    \caption{Semi-log plot of mean squared error of various estimators over the fraction of non-Gaussian marginals $\alpha \in [0, 1]$, for various marginal transforms $T$.}
    \label{fig:hat_I_nonrobust_appendix}
\end{figure*}

\section{Specific Assumptions for Estimating $H(X)$}

As shown in the main paper, to estimate the entropy of a nonparanormal distribution at the rate $O(D^2/n)$, it suffices to the univariate entropy of each variable $X_j$ at the rate $O(1/n)$. To do this, additional assumptions are required on the marginal densities $p_j$. Here, we give detailed sufficient conditions for this.

Letting $S_j \subseteq \R$ denote the support of $p_j$, the two key assumptions can be roughly classified as follows:

\begin{enumerate}[label=(\alph*),leftmargin=*]
\item
$\frac{1}{2}$-order smoothness\footnote{This is stronger than the $\frac{1}{4}$-order smoothness mandated by the minimax rate for entropy estimation~\citep{birge95densityFunctionals}, but appears necessary for most practical entropy estimators. See Section 4 of \citet{kandasamy15vonMises} for further details.}; e.g., a H\"older condition:
\[\displayindent0pt
\displaywidth\columnwidth
\sup_{x \neq y \in S_j}
      \frac{|p_j(x) - p_j(y)|}{|x - y|^{1/2}}
 < L,\]
or a (slightly weaker) Sobolev condition:
\[\displayindent0pt
\displaywidth\columnwidth
\int_{S_j} p_j^2(x) \, dx < \infty \; \text{ and }
\int_{S_j} \left( |\xi|^{1/2} |\mathcal{F} \left[ p_j \right](\xi)| \right)^2 d\xi < L,\]
(where $\mathcal{F} \left[ p_j \right](\xi)$ denotes the Fourier transform of $p_j$ evaluated at $\xi$) for some constant $L > 0$.
\item
absolute bounds $p_j(x) \in [\kappa_1,\kappa_2]$ for all $x \in S_j$ or $(a_j,b_j)$-exponential tail bounds
\[\frac{f(x)}{\exp(-a_j x^{b_j})} \in [\kappa_1,\kappa_2] \quad \text{ for all } x \in S_j\]
for some $\kappa_1, \kappa_2 \in (0, \infty)$.
\end{enumerate}

Under these assumptions, there are a variety of nonparametric univariate entropy estimators that have been shown to converge at the rate $O(1/n)$ \citep{beirlant97entropyOverview,kandasamy15vonMises,singh16kNNEntropy,moon16improving}.

\section{Lower bounding the eigenvalues of a bandable matrix}
Recall that, for $c \in (0,1)$, a matrix $\Sigma \in \R^{D \times D}$ is called \emph{$c$-bandable} if there exists a constant $c \in (0,1)$ such that, for all $i,j \in D$, $|\Sigma_{i,j}| \leq c^{|i - j|}$.

Here, we show simple bounds on the eigenvalues of a bandable correlation matrix $\Sigma$. While this result is fairly straightforward, a brief search the literature turned up no comparable results. \citet{bickel08regularized}, who originally introduced the class of bandable covariance matrices, separately assumed the existence of lower and upper bounds on the eigenvalues to prove their results. In the context of information estimation, this results of particular interest because, when $c < 1/3$ it implies a dimension-free positive lower bound on the minimum eigenvalue of $\Sigma$, hence complementing our upper bound in Theorem \ref{thm:general_I_rho_MSE_bound}.

\begin{proposition}
Suppose a symmetric matrix $\Sigma \in \R^{D \times D}$ is $c$-bandable and has identical diagonal entries $\Sigma_{j,j} = 1$. Then, the eigenvalues $\lambda_1(\Sigma),...,\lambda_D(\Sigma)$ of $\Sigma$ can be bounded as
\[\frac{1 - 3c}{1 - c} \leq \lambda_1(\Sigma),...,\lambda_D(\Sigma) \leq \frac{1 + c}{1 - c}.\]
In particular, when $c < 1/3$, we have
\[0 < \frac{1 - 3c}{1 - c} \leq \lambda_D(\Sigma).\]
\label{prop:bandable_eigenvalues}
\end{proposition}
\begin{proof}
The proof is based on the Gershgorin circle theorem~\citep{gershgorin31circleThm,varga09matrixIterativeAnalysis}. In the case of a real symmetric matrix $\Sigma$, this states that the eigenvalues of $\Sigma$ lie within a union of intervals
\begin{equation}
\left\{ \lambda_1(\Sigma),...,\lambda_D(\Sigma) \right\} \subseteq \bigcup_{j = 1}^D \left[ \Sigma_{j,j} - R_j, \Sigma_{j,j} + R_j \right],
\label{inclusion:gershgorin}
\end{equation}
where $R_j := \sum_{k \neq j} |\Sigma_{j,k}|$ is the sum of the absolute values of the non-diagonal entries of the $j^{th}$ row of $\Sigma$. In our case, since the diagonal entries of $\Sigma$ are all $\Sigma_{j,j} = 1$, we simply have to bound
\begin{align*}
\max_{j \in [D]} R_j \leq \sum_{k \neq j} c^{|k - j|}.
\end{align*}
This geometric sum is maximized when $j = \lceil D/2 \rceil$, giving
\[R_j
 \leq 2\sum_{\delta = 1}^{\lfloor D/2 \rfloor} c^{\delta}
 = 2c\frac{1 - c^{\lfloor D/2 \rfloor}}{1 - c}\
 \leq \frac{2c}{1 - c}.\]
Finally, the inclusion \eqref{inclusion:gershgorin} gives
\[\lambda_D(\Sigma) \geq 1 - \frac{2c}{1 - c} = \frac{1 - 3c}{1 - c} > 0\]
when $c < 1/3$.
$1 + \frac{2c}{1 - c} = \frac{1 + c}{1 - c}$.
\end{proof}





\end{document}